\documentclass[12pt,table]{amsart}
\usepackage{amsmath,amsfonts,amssymb,amsthm,amstext,hyperref,verbatim,lmodern,ytableau,textcomp,color,young,youngtab,tikz,subfiles,tikz-cd,bbold}
\usepackage{xcolor}
\usepackage{enumerate, verbatim,mathtools,graphicx,pgf}
\usepackage{mathrsfs}
\usepackage[mathscr]{euscript}
\usepackage[utf8]{inputenc}

\usepackage{calligra}
\usepackage{setspace,fancyhdr}
    \newtheorem{thm}{Theorem}[section]

    \newtheorem{conj}[thm]{Conjecture}
    
    \newtheorem*{conj*}{Conjecture}
\theoremstyle{definition}
    \newtheorem{definition}[thm]{Definition}

    \newtheorem{rem}[thm]{Remark}
    \newtheorem{exmp}[thm]{Example}

\numberwithin{equation}{section}
\usepackage[square,sort&compress,comma,numbers]{natbib}
\usepackage{hyperref}

\definecolor{amber}{rgb}{1.0, 0.75, 0.0}
\definecolor{ao}{rgb}{0.0, 0.50, 0.0}
\usepackage[a4paper,top=3cm,bottom=2cm,left=3cm,right=3cm,marginparwidth=1.75cm]{geometry}

\usepackage{amsmath}
\usepackage{graphicx}
\usepackage[colorinlistoftodos]{todonotes}

\title{Minimal Inversions in Integer Matrices of Fixed RSK Shape}

\author{Nimisha Pahuja}
    \address{ICTS-TIFR, Bengaluru, 560097}
    \email{nimisha.pahuja@icts.res.in}

\begin{document}

\begin{abstract}
 The Robinson-Schensted-Knuth (RSK) algorithm maps an integer matrix to a pair of semi-standard Young tableaux (SSYTs) whose underlying shape has the same integer partition. We study the set of matrices $\mathcal{M}_\lambda$ associated with a given partition $\lambda$ vis-a-vis the number of inversions of the matrix.  In the case where the integer matrix is a permutation matrix, the resulting tableaux are standard Young tableaux or SYTs. Han (EJC, 2005) combinatorially studied the set of permutations that map to SYTs of shape $\lambda$ under the RSK algorithm and counted the permutations with the minimum number of inversions in that set, as well as formulated the minimal number of inversions. Han's work can be extended to a case where the matrix is a general integer matrix and the tableaux are semi-standard Young tableaux. We have conjectured a formula for the minimal number of inversions in the set $\mathcal{M}_\lambda$ for a fixed $\lambda$. We further provide a conjecture for the characterisation of the minimal generalised matrices.

\end{abstract}

%
%
\maketitle


\section{Introduction}\label{sec:introPASEP}

A permutation has many statistics associated with it. One of the most fundamental is the length of its longest increasing subsequence. Robinson~\cite{Robinson} introduced the insertion algorithm in the study of representations of the symmetric group. Decades later, Schensted~\cite{Schensted} independently rediscovered it and connected it to the problem of longest increasing and decreasing subsequences. The resulting correspondence, now known as the Robinson-Schensted algorithm, establishes a bijection between a permutation of $N$ elements and a pair of standard Young tableaux of the same shape, where the shape is a partition of $N$.

Moreover, this partition can also be characterised in terms of increasing subsequences of the permutation, via Greene's theorem~\cite{Greene}. By a slight abuse of notation, we will refer to this partition as the shape of the permutation $\omega$. 

Knuth later generalised the Robinson-Schensted algorithm to associate a nonnegative integer matrix with a pair of semi-standard Young tableaux. This extension is known as the Robinson-Schensted-Knuth (RSK) algorithm. In this setting, the shape of a matrix is defined to be the common shape of the resulting pair of semistandard Young tableaux.

Another classic statistic of a permutation is its inversion number. An inversion of a permutation $\omega$ is a pair of entries $(\omega_i, \omega_j)$ that appear in descending order, that is, indices $i<j$ such that $\omega_i>\omega_j$.

Hohlweg~\cite{Hohlweg} studied the set of permutations of a fixed shape $\lambda$, denoted $\mathcal{S}_\lambda$, and characterised those with the minimum number of inversions. He also obtained an explicit formula for this minimum number. Han~\cite{Han} later provided a combinatorial proof of Hohlweg's result using Viennot's geometric construction involving shadow line representation of a permutation~\cite{Sagan}.

In this paper, we consider the corresponding set of matrices $\mathcal{M}_\lambda$ and aim to generalise Hohlweg's results to this setting. Since a permutation of $N$ elements can be represented as a $0$-$1$ matrix of size $N \times N$, the minimal permutations characterised by Hohlweg and Han belong to $\mathcal{M}_\lambda$, and are in fact minimal in this larger set as well. Our focus, however, is on the non-trivial case in which the size of the matrices in $\mathcal{M}_\lambda$ is as small as possible. This minimal size is equal to the number of parts of the partition $\lambda$.


We therefore fix a partition $\lambda$ with $n$ parts and restrict our attention to matrices in $M_\lambda$ of size $n \times n$. Our goal is to study the inversion statistic on this set, and in particular to determine the minimum number of inversions among matrices of shape $\lambda$.  A matrix $M \in \mathcal{M}_\lambda$ that attains this minimum will be called a \textit{minimal matrix of shape} $\lambda$. This matrix need not be unique.

This leads us to the following conjecture.

\begin{conj}\label{conj:hankel}
Let $\lambda$ be a partition with $n$ parts, and let 
$M=(m_{i,j})_{1\le i,j \le n} \in \mathcal{M}_\lambda$ 
be a minimal matrix of shape $\lambda$. Then $M$ is symmetric, that is,
\[
m_{i,j} = m_{j,i} \quad \text{for all } 1 \le i,j \le n.
\]
Moreover, every minimal matrix of shape $\lambda$ is Hankel: there exists a sequence of integers 
\[
s_2, s_3, \dots, s_{2n}
\]
such that
\[
m_{i,j} = s_{i+j}
\quad \text{for all } 1 \le i,j \le n.
\]

\end{conj}

\begin{thm}\label{thm:char}

Let $\lambda$ be a partition with $n$ parts, and let 
$r_i = \lambda_i - \lambda_{i+1}$ for $1 \le i < n$, and 
$r_n = \lambda_n$. 
Thus $r_i$ denotes the number of columns of height $i$ in $\lambda$.

Define
\[
a_i = \left\lfloor \frac{r_i}{2} \right\rfloor,
\qquad
b_i = \left\lceil \frac{r_i}{2} \right\rceil,
\]
or vice versa.

Assume Conjecture~\ref{conj:hankel} holds, so that every minimal matrix is Hankel. 
Let $M=(m_{i,j})$ be a minimal Hankel matrix of shape $\lambda$, 
and write $m_{i,j}=s_{i+j}$. Then
\[
s_i = a_{i-1}
\quad \text{and} \quad
s_{2n-i} = b_{i-1}
\qquad
\text{for } 2 \le i \le n.
\]

\end{thm}

The conjectured form of the parameters $s_i$ reflects the Hankel structure imposed on minimal matrices. Since the multiplicities $r_i$ describe the number of columns of height $i$ in the Young diagram of $\lambda$, they encode how the weight of a matrix (sum of all the entries) must be distributed across anti-diagonals to obtain the prescribed shape.

Under the fixed shape restriction, contributions to the inversion number coming from opposite anti-diagonals are naturally paired, because of the anti-diagonals being of the same length. Minimising inversions, therefore, suggests distributing the total mass $r_i$ as evenly as possible between the corresponding positions $s_{i+1}$ and $s_{2n-i+1}$. This leads to the balanced floor–ceiling splitting, which heuristically minimises the inversion count.
To illustrate the conjectured structure, suppose that $\lambda$ has $5$ parts. If $M$ is a minimal matrix of shape $\lambda$, then according to Conjecture~\ref{conj:hankel}, it is Hankel. Hence $M=(m_{i,j})$ satisfies $m_{i,j}=s_{i+j}$, where the sequence 
$(s_2,s_3,\dots,s_{10})$ is determined by the splitting of the column 
multiplicities $r_i$.

In this case, the matrix has the form
\[
 \begin{bmatrix}
 a_1 & a_2 & a_3 & a_4 & s_6 \\
 a_2 & a_3 & a_4 & s_6 & b_1 \\
 a_3 & a_4 & s_6 & b_1 & b_2 \\
 a_4 & s_6 & b_1 & b_2 & b_3 \\
 s_6 & b_1 & b_2 & b_3 & b_4
\end{bmatrix},\]
where for $1 \le i \le 4$ we have 
\[
\{a_i,b_i\}
=
\left\{
\left\lfloor \frac{r_i}{2} \right\rfloor,
\left\lceil \frac{r_i}{2} \right\rceil
\right\},
\]
and the central anti-diagonal entry $s_6$ corresponds to $r_5=\lambda_5$.


The paper is organised as follows. In Section~\ref{sec:prelim}, we recall the necessary background on the Robinson–Schensted–Knuth correspondence, inversion statistics, and related combinatorial tools. In Section~\ref{sec:structure}, we introduce the class of matrices $\mathcal{M}_{\lambda}$ and establish the properties that will be used throughout the paper. Section~\ref{sec:main} contains the proof of our main results, including the determination of the minimal inversion number and the characterisation of minimal matrices. 




\section{Preliminaries and Examples}\label{sec:prelim}
We first introduce the basic concepts and notations that will be used throughout the paper. A \emph{permutation} of a set $S$ is a bijection from the set $S$ to itself. Consider $S=\{1,2,\ldots,N\}= [N]$, then a permutation $\omega$ of $S$ which takes $i \rightarrow \omega_i$ can  be written as $\omega_1\omega_2\cdots \omega_N$ in one-line notation. The set of all such permutations is denoted by $S_N$. A permutation $\omega \in S_N$ can also be represented as a $0-1$ matrix $P=(p_{ij})$ of size $N \times N$ where 
\[p_{ij}=\begin{cases}
  1, & j=\omega_i\\
0, & \text{otherwise}  
\end{cases}
\]

A \emph{partition} of $N \in \mathbb{N}$ is defined to be a sequence $\lambda=(\lambda_1,\lambda_2, \ldots,\lambda_\ell)$ satisfying $\sum \lambda_i = N$ and $\lambda_1 \geq \lambda_2 \geq \cdots \geq \lambda_n > 0$.
We write $\lambda \vdash N$ or $N=|\lambda|$ to denote a partition of $n$. $\lambda_i$'s are called parts of $\lambda$.  The number of non-zero parts in $\lambda$ is denoted by $\ell(\lambda)$. For example, $\lambda=(4,2,2,1) \vdash 9$, and $\ell(\lambda)=4$. 

The \emph{Young's diagram} of $\lambda$ is a left-justified array of n boxes with $\lambda_i$ boxes in $i^{th}$ row $i$ , for $1 \leq i \leq \ell$. For example, for the partition $\lambda=(4,2,2,1) \vdash 8$, the Young diagram is given as

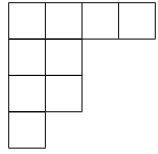
\begin{center}{$\yng(4,2,2,1)$}\\
				$(4,2,2,1)$.
\end{center}

The conjugate of a partition $\lambda$, denoted by $\lambda'$, is the partition whose diagram is the transpose of the diagram of $\lambda$. The conjugate of the partition $(4,2,2,1)$ is $(4,3,1,1)$.

\begin{definition}\label{def:tab}
Let $\lambda \vdash n$. A \emph{Young tableau of shape $\lambda$,} is an array $t$ obtained by filling the boxes in a Young diagram with positive integers. Let $t_{i,j}$ denote the entry in position $(i,j)$ of $t$. We write sh$(t)=\lambda$ .
\end{definition}

A Young tableau is called \emph{standard} if the entries increase strictly along each row and column. If the entries increase weakly across each row, and strictly down each column, then the tableau is known as a \textit{semi-standard Young tableau} (SSYT). 
 
\begin{exmp} Let $\lambda=(4,2,2,1) \vdash 9$ be a partition, here are a few Young tableaux of shape $\lambda$. The first two are examples of semi-standard, and the last is a standard Young tableau.
\end{exmp}
\begin{center}
$\begin{array}{ccc}
\begin{Young}
1 & 1 & 1 & 1\cr
2 & 2\cr
3 & 3\cr
4\cr
\end{Young}\hspace{0.8in}  &
\begin{Young}
1 & 2 & 2 & 4\cr
2 & 4\cr
6 & 7\cr
8\cr
\end{Young}\hspace{0.8in}  &
\begin{Young}
1 & 2 & 4 & 5\cr
3 & 6\cr
7 & 9\cr
8\cr
\end{Young}
\end{array}$
\end{center}

\subsection{Subsequences}
An increasing subsequence of a permutation $\omega=\omega_1\omega_2\ldots\omega_N$is a sequence \[\omega_{i_1} < \omega_{i_2} < \ldots < \omega_{i_k}\] with  $i_1 < i_2 < \ldots < i_k$. The \textit{length} of a subsequence is the number of its elements. An increasing subsequence of maximal length is called a \textit{longest increasing subsequence} (LIS). Analogously, one can define a \textit{decreasing subsequence} and a \textit{longest decreasing subsequence} (LDS). For example, if $\omega = 4\, 6\, 8\, 1\, 3\, 7\, 2\, 5\, 9$, then $6\,7\,9$ and $1\,3\,7\,9$ are increasing subsequences, while $6\,5$ and $6\,3\,2$ are decreasing subsequences.

The lengths of the longest increasing and decreasing subsequences 
of a permutation $\omega$ are intimately related to the shape 
$\mathrm{sh}(\omega)$ via Schensted’s theorem, which states that

\begin{thm}[Schensted 61]\label{th: sch}\cite{Schensted}
For every $\omega \in S_n$, the length of the longest increasing subsequence of $\omega$ is the length of the first row of $P(\omega)$. The length of the longest decreasing subsequence of $\omega$ is the length of the first column of $P(\omega)$.
\end{thm}

 \begin{definition} Let $\omega$ be a permutation. A subsequence $\sigma$ of $\omega$ is said to be \textit{$k$-increasing} if, as a set, it can be written as a disjoint union
 \[\sigma = \sigma_1 \uplus \sigma_2 \uplus \cdots \uplus \sigma_k,\] 
where the $\sigma_i$ are increasing subsequences of $\omega$. \end{definition}

\begin{exmp} If $\omega =  4\, 6\, 8\, 1\, 3\, 7\, 2\, 5\, 9$, then $\sigma = 4\,7\,9 \uplus 1\,2\,5$ is a $2$-increasing sequence.\end{exmp}
Similarly, a subsequence $\sigma$ is called $k$-decreasing if it can be written as a disjoint union of $k$ decreasing subsequences. 
Let $i_k(\omega)$ denote the maximum length of a $k$- increasing subsequence of $\omega$, and let $d_k(\omega)$ denote the corresponding maximum length of a $k$-decreasing subsequence. The following theorem, due to given by Curtis Greene~\cite{Greene}, extends Schensted's Theorem. 
\begin{thm}[Greene's Theorem]\cite{Greene}\label{thm:gr}
Let $\omega \in S_n$, let $\mathrm{sh}(\omega)=(\lambda_1, \lambda_2, \cdots, \lambda_l)$ with conjugate partition $(\lambda'_1, \lambda'_2, \cdots , \lambda'_m)$. Then for every $k \geq 1$, \begin{equation*} i_k(\omega)=\lambda_1+\lambda_2 +\ldots \lambda_k\end{equation*}
\begin{equation*} d_k(\omega)=\lambda'_1+\lambda'_2 +\ldots \lambda'_k \end{equation*}
\end{thm}

\begin{rem}
    Observe that Greene's theorem does not provide an interpretation of the individual terms $\lambda_k$ and $\lambda'_k$.  In particular, it is not generally true that a $k$-increasing subsequence of maximal length can be obtained by extending a maximal $(k-1)$-increasing subsequence by $\lambda_k$ additional elements. For example, let $\omega = 247951368.$ The shape of this permutation is $(5,3,1)$, so that $\lambda_1 = 5$ and $\lambda_2 = 3$. Thus $i_1(\omega) = 5$ and $i_2(\omega) = 8$. Although $24568$ is an increasing subsequence of maximal length $5$, it cannot be extended by an increasing sequence of length $3$ to form a longest $2$-increasing subsequence. In fact, there is a unique $2$-increasing subsequence of length $8$ given by $ \sigma = 2479 \uplus 1368$, which can only be partitioned into two increasing blocks of size $4$. 
\end{rem}





\section{Matrices of Fixed Shape under RSK}\label{sec:structure}

To relate matrices to the Robinson–Schensted–Knuth correspondence, it is convenient to encode a matrix as a two-line array. This representation allows us to view matrices as weighted analogues of permutations and makes the insertion procedure underlying RSK transparent. We therefore recall the notion of a generalised permutation.
 
A \textit{generalised permutation} is a  two-line array 
\[\omega = \left(\begin{array}{ccccc}
 i_1 & i_2 & i_3 & \ldots & i_N\\
j_1 & j_2 & j_3 & \ldots & j_{N}
\end{array}\right)\]
such that
 \begin{itemize}
		\item $i_1 \leq i_2 \leq \cdots \leq i_N$, and
		\item whenever $i_r=i_s$ with $r \leq s$, we have $j_r \leq j_s$. \\ 
		\end{itemize}

In other words, the top row is weakly increasing, and two pairs with the same top entry have bottom entries in weakly increasing order. 

A matrix $M=(m_{i,j})$ of size $n \times n$ can be written in a two-line notation as a generalised permutation by listing each pair $(i,j)$ exactly $m_{i,j}$ times. That is, we expand each entry of the matrix into repeated pairs.
For example: 
\[\begin{bmatrix}
1 & 0 & 2 \\
0 & 2 & 0 \\
1 & 1 & 0
\end{bmatrix} \text{ correspondes to }
\left(\begin{array}{ccccccc}
				1 & 1 & 1 & 2 & 2 & 3 & 3 \\
				1 & 3 & 3 & 2 & 2 & 1 & 2
				\end{array}\right)\]

We now define inversion and increasing statistics for non-negative integer matrices in a way that naturally extends the corresponding notions for permutations. 

\begin{definition}
    An \textit{inversion} of a matrix $M \in \mathcal{M}_\lambda$ is a pair of positions $(i,j)$ and $(k,l)\}$ such that $m_{ij},m_{kl}>0$, and $i<k$ and $j>l$.
\end{definition}
The total number of inversions in $M$ is therefore
\[
\sum_{i<k,\; j>l} m_{i,j} m_{k,l}.
\]

A sequence in a matrix $M$ is called increasing if the corresponding pairs in its associated generalised permutation form a weakly increasing sequence in both rows. Equivalently, an increasing sequence corresponds to a lattice path in $M$ that moves weakly to the right and weakly downward, and its length is obtained by summing the entries along the path.

A longest increasing subsequence of a matrix may thus be interpreted as a maximal-weight right–down path. By Greene’s theorem, the shape of a matrix can be recovered from the lengths of its longest $k$-increasing subsequences, for $1\leq k \leq n$.

\subsection{Robinson-Schensted-Knuth Algorithm}
The Robinson-Schensted-Knuth (RSK) algorithm constructs a bijection between a non-negative integer matrix $M$ and a pair of semi-standard Young tableaux $(P,Q)$, of the same shape $\lambda$. The bijection is denoted by 
\[M \leftrightarrow (P,Q).\] 
Equivalently, since a matrix may be encoded as a generalised permutation (biword), RSK may be described as an insertion algorithm applied to this two-line array.

\textbf{Insertion}
Let $P$ be a tableau whose rows are weakly increasing and whose columns are strictly increasing. Let $x_1$ be a positive integer. The insertion of $x_1$ into $P$ proceeds as follows:
Compare $x$ with the entries in the first row $R_1$ of $P$. If $x_1$ is greater than or equal to all entries in $R_1$, append $x_1$ at the end of $R_1$. Otherwise, let $x_2$ be the smaller entry in $R_1$ that is strictly greater than $x_1$. Replace $x_2$ by $x_1$, and the bumped entry $x_2$ is then inserted into the second row $R_2$. This bumping process continues row by row until an entry is added at the end of some row $R_s$. The resulting tableau remains semistandard.


Given $\omega$ of size $N$, we construct a sequence of tableau pairs
\begin{center}
$(P_0,Q_0)=(\emptyset,\emptyset), (P_1,Q_1),(P_2,Q_2),\ldots,(P_N,Q_N)=(P,Q),$
\end{center}
inductively, such that at each step sh ($P_k$) = sh ($Q_k), \, \forall \,1 \leq k \leq n$. We write $RSK(\omega)=(P,Q)$.

 Assume that $(P_{k-1},Q_{k-1})$ has been constructed.The RSK algorithm proceeds inductively through the insertion operation, where at each iteration, we obtain $P_k$ by inserting $j_k$ into $P_{k-1}$, and recording $i_k$ in $Q_{k-1}$ in the cell $(s,t)$ where the new cell is added in $P_{k}$.

\begin{exmp}
The following example illustrates the bijection in the $3\times 3$ case. Consider the matrix
\[
M =
\begin{bmatrix}
1 & 1 & 0 \\
0 & 2 & 1 \\
1 & 0 & 1
\end{bmatrix}.
\]
Its associated generalised permutation is
\[
\omega =
\begin{pmatrix}
1 & 1 & 2 & 2 & 2 & 3 & 3 \\
1 & 2 & 2 & 2 & 3 & 1 & 3
\end{pmatrix}.
\]

Applying the RSK correspondence, we obtain the pair of semistandard Young tableaux
\[
P =
\begin{ytableau}
1 & 1 & 2 & 3 \\
2 & 3 \\
3
\end{ytableau}\,
\qquad
Q =
\begin{ytableau}
1 & 1 & 2 & 2 \\
2 & 3 \\
3
\end{ytableau}.
\]
Both tableaux have shape $(4,2,1)$.
\end{exmp}




\subsection*{Some consequences of the RS Algorithm}
\begin{thm}[Schutzenberger 63]\rm \cite{Sagan} \em If $\omega^{-1}$ denotes the inverse of $\omega$, then $P(\omega^{-1}) = Q(\omega)$ and $Q(\omega^{-1}) = P(\omega)$.
\end{thm}
 This result was extended by Knuth~\cite{Knuth} for the case of non-negative integer matrices to say that if $M \leftrightarrow (P, Q)$, then $M^T \leftrightarrow (Q, P)$, where $M^T$ refers to the transpose of $M$. 

\begin{thm}[Schensted 61]\label{th:rev}\rm \cite{Sagan}\em Let $\omega^r$ be the reversal of $\omega$, i.e. if $\omega = x_1x_2\cdots x_n$, then $\omega^r= x_nx_{n-1}\cdots x_1$. \\
If $P(\omega)=P$, then $P(\omega^r) = P^t$, the transpose of $P.$
\end{thm}


\begin{definition} The \textit{reading word} or \textit{row-word} of a tableau $P$ is the permutation \begin{center} $\omega_P = R_\ell R_{\ell-1}\cdots R_1$ \end{center} where $R_1,\cdots R_\ell$ are the rows of $P$.\end{definition}



\section{Main results}\label{sec:main}
We prove conjecture~\ref{conj:hankel} for the case where the fixed shape $\lambda$ is a two-row partition.
\begin{thm}
    Let $\lambda$ be a partition with $n$ parts, where $n =2$ and let $M=(m_{i,j})$ 
be a minimal matrix of shape $\lambda$. Then $M$ is symmetric and a Hankel matrix.
Moreover,  \[M=
\begin{bmatrix}
    s_2 & \lambda_2\\
    \lambda_2 & s_4,
\end{bmatrix}
\]
where $s_2=k$ and $s_4=\lambda_1-\lambda_2-k$ for any $k \in \{0,1,\ldots,\lambda_1-\lambda_2$. 

\end{thm}
\begin{proof}
    Let $\lambda=(\lambda_1,\lambda_2)$ be a two-row partition, $M \in \mathcal{M}_\lambda$ be a minimal matrix. Assume
    \[M=\begin{bmatrix}
    a & b\\
    c & d
\end{bmatrix},
\] is non-Hankel, i.e., $b>c$ (without loss of generality). Thus, the longest increasing subsequence has length $a+b+d$. By Greene's theorem~\ref{thm:gr}, the shape of $M$ is given by 
\[(a+b+d,c).\] Here, $c>0$ because $\lambda$ is a two-row partition. and $M$ has a total of $bc$ inversions.
Consider the matrix
\[M'=\begin{bmatrix}
    a+(b-c) & c\\
    c & d
\end{bmatrix},
\] 
As $M'$ is symmetric, using Theorem~\ref{thm:gr} again, we have the shape of $M$ to be $\lambda_1,\lambda_2$. Thus, $M' \in \mathcal{M}_\lambda$ but has $c^2$ inversions, which is strictly less than $bc$. This contradicts our assumption that $M$ is minimal.
Thus, any minimal matrix of size $2 \times 2$ must be Hankel with parameters $s_2, s_3=\lambda_2$, and $s_4$, where $s_2+s_4=\lambda_1-\lambda_2$. This equation has $\lambda_1-\lambda_2+1$ solutions as the $(1,1)$ and $(2,2)$ pairs do not contribute to any inversions. The symmetry follows from the Hankel property.

 \end{proof}

Since both the RSK shape and the inversion number are preserved under transposition, the set of minimal matrices is closed under transpose. Thus, there is no a priori reason for a minimal configuration to prefer one direction over the other. Numerical computations for small shapes indicate that minimal configurations tend to distribute the weight symmetrically across the matrix. Together, these ideas suggest that minimal matrices should exhibit a highly symmetric structure, depending only on the sum $i+j$. This leads to the main conjecture.

\begin{conj*}[Conjecture~\ref{conj:hankel}]
Let $\lambda$ be a partition with $n$ parts, and let 
$M=(m_{i,j})_{1\le i,j \le n} \in \mathcal{M}_\lambda$ 
be a minimal matrix of shape $\lambda$. Then $M$ is symmetric, that is,
\[
m_{i,j} = m_{j,i} \quad \text{for all } 1 \le i,j \le n.
\]
Moreover, every minimal matrix of shape $\lambda$ is Hankel, that is, there exists a sequence of integers 
\[
s_2, s_3, \dots, s_{2n}
\]
such that
\[
m_{i,j} = s_{i+j}
\quad \text{for all } 1 \le i,j \le n.
\]

\end{conj*}

Assuming Conjecture~\ref{conj:hankel}, we now prove theorem~\ref{thm:char}.
\begin{proof}[Proof of Theorem~\ref{thm:char}]
Consider a minimal matrix $M$. Assuming Conjecture~\ref{conj:hankel}, $M$ is a Hankel matrix with associated anti-diagonal parameters 
\[
(s_2, s_3, \dots, s_{2n}).
\]
Thus $m_{i,j} = s_{i+j}$.
Any right-down path in $M$ from $(1,1)$ to $(n,n)$ intersects 
each anti-diagonal exactly once. Therefore, every such path has 
the same weight,
\[
s_2 + s_3 + \cdots + s_{2n}.
\]
Consequently, a longest increasing subsequence of $M$ has length $s_2 + s_3 + \cdots + s_{2n}$. Because of the Hankel structure of the matrix, a longest $2$-increasing subsequence in $M$ is indeed the union of a longest increasing subsequence with the second longest subsequence obtained after deleting the entries from this LIS.
More generally, a longest $k$-increasing subsequence of $M$ corresponds to a union of $k$ non-intersecting maximal right-down paths. Applying Greene’s theorem, we obtain the system
\[
\begin{aligned}
s_2 + s_3 + s_4\cdots + s_{2n} &= \lambda_1, \\
s_3 + s_4 + \cdots + s_{2n-1} &= \lambda_2, \\
&\ \ \vdots \\
s_{n} + s_{n+1} + s_{n+2} &= \lambda_{n-1}, \\
s_{n+1} &= \lambda_n.
\end{aligned}
\]
Solving this system, we get \[s_{k+1}+s_{2n-k+1}=\lambda_k-\lambda_{k+1}.\]
The quantity on the right-hand side is exactly equal to $r_k$, which is the number of columns of height $k$ in $\lambda$. Looking at the inversions number, the formula becomes a sum of quadratic expressions in $s_k$'s for $k \in \{3,4,\ldots,2n-1\}$. Inversions are minimised when they are as close to $\frac{r_k}{2}$ as possible. But since $s_k$ is an integer for all $k$, we have 

\[
\{s_{k+1},\, s_{2n-k+1}\}
=
\left\{
\left\lfloor \frac{r_k}{2} \right\rfloor,
\left\lceil \frac{r_k}{2} \right\rceil
\right\},
\]
for each $1 < k < n-1$, and $s_{n+1}=r_n$.

Finally, $s_2+s_{2n}=\lambda_1-\lambda_2=r_1$ implies there are $r_1$ pairs with $(1,1)$ or $(n,n)$ and these pairs do not contribute to the inversions of the matrix. Thus, these $r_1$ columns can be distributed between $s_2,s_{2n}$ in $r_1+1$ ways.
\end{proof}

Before computing the number of inversions in minimal matrices, it is useful to examine the semistandard Young tableaux to which they correspond under RSK. If M is symmetric, then the associated tableaux satisfy $P=Q$.  As an example, consider the $3 \times 3$ Hankel matrix with parameters $s_i=2$ for $3 \leq i \leq 5$, and $s_2=0, s_6=4$, that is, 
\[M=\begin{bmatrix}
    0 & 2 & 2 \\
    2 & 2 & 2\\
    2 & 2 & 4.
\end{bmatrix}
\]
The corresponding tableaux look like
\[P=Q=
\begin{ytableau}
1 & 1  & 1  & 1 & 2 & 2 & 3 & 3 & 3 & 3 \\
2 &  2 & 2  & 2 & 3 & 3\\
3 & 3   
\end{ytableau}
\]
In general, for a Hankel matrix with parameters $(s_2,\dots,s_{2n})$, 
the tableaux $P$ and $Q$ consist of:
\begin{itemize}
    \item $s_i$ columns of height $i-1$ filled with the entries $1,2,\dots,i-1$, for $2 \le i \le n+1$, and
    \item $s_i$ columns of height $2n-i+1$ filled with the entries 
    $i-n, i-n+1,\dots,n$, for $n+1 \le i \le 2n$.
\end{itemize}
We refer to the first type as \emph{forward columns} and the second as 
\emph{backward columns}.

Thus, the column multiplicities of the tableaux are determined directly 
by the anti-diagonal parameters $s_i$, and hence by the conjugate partition 
$\lambda'$.

Assume that minimal matrices are Hankel and that the anti-diagonal 
parameters split according to the floor–ceiling rule. 
Then the associated tableaux $P=Q$ are completely determined by the 
column multiplicities encoded by $\lambda'$.

Inversions may be interpreted as interactions between columns of the tableau. 
First, for each column $i$, $\lambda'_i$ gives the height of column $i$. Each column internally contributes $\binom{\lambda'_i}{2}$ inversions among themselves. 

Secondly, additional inversions arise from interactions between entries 
in different columns. 
Two columns of the same type (both forward or both backward) contribute 
$\binom{u}{2}$ inversions, where $u$ is the length of the weakly shorter column. Each column $i$ occurs in such kind of pairing exactly $\lfloor \frac{i}{2} \rfloor$ times.

For the interaction between a forward and a backward column, we look at the overlap of the two. For a column of length $\lambda_i$ filled with the entries $1,2,\dots,\lambda_i-1$, and a column of length $\lambda_j$ filled with entries
    $n-\lambda_j+1,\ldots,n$, the overlapping portion 
of the corresponding column sets has size $\max(0,\lambda'_i + \lambda'_j - n)$, 
and each such overlap contributes 
$\binom{\lambda'_i + \lambda'_j - n}{2}$ inversions. Because of the floor-ceiling splitting, we can consider columns as alternating between the forward and backward type in a tableau associated to a minimal matrix. Thus, we consider the overlap to be happening between columns with odd and even indices. Summing over all such pairs yields the second term.
Combining these contributions gives the stated formula.

\begin{thm}
Let $\lambda$ be a partition and let $\lambda'$ denote its conjugate. 
The minimum number of inversions among words of shape $\lambda$ is given by
\[ \sum_{i \ge 1}
\left( \left\lfloor \frac{i}{2} \right\rfloor + 1 \right)
\binom{\lambda'_i}{2}
\;+\;
\sum_{\substack{i \ge 1 \\ i \text{ odd}}}
\;\sum_{\substack{j \ge 2 \\ j \text{ even}}}
\binom{\lambda'_i + \lambda'_j - n}{2},
\]
where $n = \ell(\lambda)$ is the number of parts of $\lambda$.
\end{thm}

\end{document}